\documentclass[12pt]{amsart}
\usepackage{amssymb}
\usepackage{amsmath,amsfonts,amssymb,upgreek}

\usepackage[breaklinks]{hyperref}
\usepackage{amscd}
\usepackage{mathtools}

\theoremstyle{thmrm}
\usepackage{tikz}
\usepackage{enumerate}
\usetikzlibrary{calc,decorations.markings}
\usepackage{tikz-cd}
\usepackage{graphicx}

\theoremstyle{plain}
\newtheorem{no}{Notation}[section]

\newtheorem{thm}{Theorem}[section]
\newtheorem{lemma}[thm]{Lemma}
\newtheorem{prop}[thm]{Proposition}

\newtheorem{cor}[thm]{Corollary}

\newtheorem{question}[thm]{Question}
\newtheorem{defn}[thm]{Definition}
\newtheorem{remark}[thm]{Remark}
\newtheorem{example}[thm]{Example}
\newtheorem*{ack}{Acknowledgement}

\newcommand{\Sym}{\operatorname{Sym} }
\newcommand{\op}{\operatorname{op} }
\newcommand{\GL}{\operatorname{GL} }

\newcommand{\Z}{\operatorname{Z} }
\newcommand{\Cl}{\operatorname{Cl} }

\newcommand{\Fix}{\operatorname{Fix}}

\newcommand{\Aut}{\operatorname{Aut} }
\newcommand{\ird}{\operatorname{ird} }

\newcommand{\Ker}{\operatorname{Ker} }

\newcommand{\Ann}{\operatorname{Ann} }

\numberwithin{equation}{section}

\begin{document}
	
	\title{Representations of skew braces}
	
	\author{Nishant Rathee}
\address{ DEPARTMENT OF MATHEMATICS, NATIONAL INSTITUTE OF TECHNOLOGY, TIRUCHIRAPPALLI-620015, INDIA}
\email{nishant@nitt.edu, monurathee2@gmail.com}		
	
	
	\author{Ayush Udeep}
	\address{Department of Mathematical Sciences, Indian Institute of Science Education and Research (IISER) Mohali, Sector 81, SAS Nagar, P O Manauli, Punjab 140306, India} 
	\email{udeepayush@gmail.com}
	
	\subjclass[2010]{16T25, 20C15}
	\keywords{Isoclinism; representation; skew left brace; Yang--Baxter equation}

\begin{abstract}
In this paper, we explore linear representations of skew left braces, which are known to provide bijective non-degenerate set-theoretical solutions to the Yang--Baxter equation that are not necessarily involutive.  A skew left brace $(A, \cdot, \circ)$ induces an action $\lambda^{\op}: (A, \circ) \to \Aut (A, \cdot)$, giving rise to the group $\Lambda_{A^{\op}} = (A, \cdot) \rtimes_{\lambda^{\op}} (A, \circ)$. We prove that if $A$ and $B$ are isoclinic skew left braces, then $\Lambda_{A^{\op}}$ and $\Lambda_{B^{\op}}$ are also isoclinic under some mild restrictions on the centers of the respective groups. Our key observation is that there is a one-to-one correspondence between the set of equivalence classes of irreducible representations of $(A, \cdot, \circ)$ and that of the group $\Lambda_{A^{\op}}$. We obtain a decomposition of the induced representation of the additive group $(A, \cdot)$ and of the multiplicative group $(A, \circ)$ corresponding to the regular representation of the group $\Lambda_{A^{\op}}$. As examples, we compute the dimensions of the irreducible representations for several skew left braces with prime power orders.
\end{abstract}
\maketitle

\section{Introduction}

The Yang--Baxter equation is a fundamental equation in mathematical physics that arises in the study of integrable systems, quantum information theory, exactly solvable models in statistical mechanics and quantum field theory. From a mathematical perspective, the equation has deep connections to knot theory, braid theory and Hopf algebras, to name a few. It has been an ongoing program to classify set-theoretical solutions to this ubiquitous equation. Interestingly, large families of set-theoretical solutions to this equation arise from algebraic structures. Rump \cite{WR07} introduced left braces as generalisations of Jacobson radical rings and proved that they give rise to involutive set-theoretical solutions to the Yang--Baxter equation. Later, Guarnieri and Vendramin \cite{GV17} extended this concept to skew left braces, which provide bijective non-degenerate set-theoretical solutions to the Yang--Baxter equation that are not necessarily involutive. 
\par 

Formally, a skew left brace is a triple \((A, \cdot, \circ)\), where \((A, \cdot)\) and \((A, \circ)\) are groups such that
$$
a \circ (b \cdot c) = (a \circ b) \cdot a^{-1} \cdot (a \circ c)
$$
 for all $a, b, c \in A$, where \(a^{-1}\) denotes the inverse of \(a\) in \((A, \cdot)\). In a recent work  \cite{LV23}, Letourmy and Vendramin introduced linear representations of skew left braces to investigate their Schur covers and lifting property of their annihilator extensions. Their motivation for defining representations of skew left braces can be traced back to \cite{Zhu}. Representations of skew left braces can be used to understand and classify these objects, and hence classify solutions to the Yang-Baxter equation. In addition, skew left braces are intriguing algebraic structures positioned between groups and rings, and hence it is of independent interest to explore how these objects act on vector spaces.
\par 

A skew left brace $(A, \cdot, \circ)$ gives a group homomorphism $\lambda: (A, \circ) \to \Aut (A, \cdot)$ given by 
$$\lambda_a(b)=a^{-1} \cdot (a \circ b)$$
 for all $a, b \in A$. This homomorphism, called the lambda map, captures the essential properties of the skew left brace. Given a skew left brace $(A, \cdot, \circ)$, we can define its opposite skew left brace $(A, \cdot^{\op}, \circ)$, where $(A, \cdot^{\op})$ is the opposite group of $(A, \cdot)$  \cite{KT20}. The lambda map $\lambda^{\op}$ of  $(A, \cdot^{\op}, \circ)$ satisfies the property that $\lambda^{\op}: (A, \circ) \rightarrow (A, \cdot)$ is also a group homomorphism. This gives the semi-direct product 
$$
\Lambda_{A^{\op}} = (A, \cdot) \rtimes_{\lambda^{\op}} (A, \circ).
$$
It has been shown in \cite{KT24} that representations of the group $\Lambda_{A^{\op}}$ play a crucial role in the representation theory of the skew left brace $(A, \cdot, \circ)$. In fact, a representation of $(A, \cdot, \circ)$ naturally gives a representation of the group $\Lambda_{A^{\op}}$, an observation credited to Letourmy and Vendramin. This leads to analogues of Maschke's and Clifford's theorem for skew left braces  \cite{KT24}. In this paper, we observe that there is a one-to-one correspondence between skew left brace representations of $(A, \cdot, \circ)$ and  group representations of $\Lambda_{A^{\op}}$. It turns out that the group $\Lambda_{A^{\op}} $ reflects many structural properties of the skew left brace $(A, \cdot, \circ)$. Since representations of various classes of groups are well-studied, we have an efficient tool to explore representations of skew left braces. 
\par

The paper is organised as follows. Section \ref{sec:prelims} introduces the necessary preliminaries on skew left braces. Section \ref{sec:properties} is dedicated to exploring the basic properties of the group $\Lambda_{A^{\op}}$. We explicitly compute the center and commutator of $\Lambda_{A^{\op}}$ (Proposition \ref{centersdp} and Theorem \ref{commutator1}). We prove that if $A$ and $B$ are isoclinic skew left braces, then the commutators of $\Lambda_{A^{\op}}$ and $\Lambda_{B^{\op}}$ are isomorphic (Theorem \ref{isocom}). Furthermore, $\Lambda_{A^{\op}}$ and $\Lambda_{B^{\op}}$ are also isoclinic under some mild restrictions on the centers of the respective groups (Theorem \ref{isoclinicsdp}). Section \ref{sec:Rep} focuses on the representation theory of skew left braces. Using Theorem \ref{thm:grouptoskewbrace}, we make a key observation that there is a one-to-one correspondence between the set of equivalence classes of irreducible representations of a skew left brace $(A, \cdot, \circ)$ and that of its corresponding group $\Lambda_{A^{\op}}$. We compute the dimensions of the irreducible representations for several skew left braces of prime power orders. In Theorem \ref{thm:regularrep}, we obtain a decomposition of the induced representation of the additive group $(A, \cdot)$ and of the multiplicative group $(A, \circ)$ corresponding to the regular representation of the group $\Lambda_{A^{\op}}$. In Section \ref{conjugacy and linearity}, we derive some bounds on the number of inequivalent irreducible representations of  a finite skew left brace (Theorem \ref{number of rep bound}). We conclude by discussing linearity of skew left braces.
\medskip

\section{Preliminaries on skew braces} \label{sec:prelims}
In this section, we revisit basic definitions on  skew left braces that we shall use in later sections (see \cite{GV17, KT20}). 

\begin{defn}
	A skew left brace is a  triple $(A,\cdot ,\circ)$, where $(A,\cdot)$ and $(A, \circ)$ are groups  such that
	$$a \circ (b \cdot c)=(a\circ b) \cdot a^{-1} \cdot (a \circ c)$$
	 for all $a,b,c \in A$, where $a^{-1}$ denotes the inverse of $a$ in $(A, \cdot)$. The groups $(A,\cdot)$ and $(A, \circ)$ are called the additive and the multiplicative groups of $(A,\cdot ,\circ)$, respectively.  Furthermore, if  $(A,\cdot)$ is an abelian group, then $(A,\cdot ,\circ)$ is called a left brace.
\end{defn}

The inverse of $a \in A$ in the multiplicative group $(A, \circ)$ will be denoted by $a^\dagger$.	Any group $(A, \cdot)$ can be turned into a skew left brace with the same  `$+$' and `$\circ$' operations, and is referred as a {\it trivial skew left brace}.
\par

Given a skew left brace $(A,\cdot ,\circ)$, it is a direct check to see that the map $\lambda: (A, \circ) \rightarrow \operatorname{Aut}(A, \cdot)$ defined by $\lambda_a(b) = a^{-1} \cdot (a \circ b)$ for $a,b \in A$, is a group homomorphism. The map $\lambda$ is called the {\it lambda map} of $(A,\cdot ,\circ)$.

\begin{defn}
Let $(A, \cdot, \circ)$ be a skew left brace. Then $(A, \cdot^{\op}, \circ)$ is also a skew left brace, called the opposite skew left brace of $A$, where $a \cdot^{\op} b := b \cdot a$ for all $a, b \in A$.
\end{defn}

It is worth noting that the lambda map  $\lambda^{\op}$ of $(A, \cdot^{\op}, \circ)$ is given by
$$
	\lambda^{\op}_a(b) = (a \circ b) \cdot a^{-1} = a \cdot \lambda_a(b) \cdot a^{-1},
$$
	and  $\lambda^{\op}:(A, \circ)\to \Aut(A, \cdot)$ is also a group homomorphism.

\begin{defn}
A skew left brace  $(A, \cdot, \circ)$ is called self-opposite if $(A, \cdot, \circ)$ and  $(A, \cdot^{\op}, \circ)$ are isomorphic as skew left braces.
\end{defn}

For instance, all left braces are self-opposite.

\begin{defn}
A subset $I$ of $A$ is called an ideal of $(A, \cdot, \circ)$ if it is a normal subgroup of both $(A,\cdot)$ and $(A,\circ)$, and $\lambda_a(I)\subseteq I$ for all $a\in A$. \par 
\end{defn}

It is clear that if $I$ is an ideal, then $(I,\cdot,\circ)$ is also a skew left brace. Further,  $a\cdot I = a\circ I$ for all $a\in A$, and hence we can form the quotient skew left brace $A/I$.  For a group $(G, \cdot)$, we denote its center and commutator subgroup by $\Z(G, \cdot)$ and $(G, \cdot)'$, respectively.

\begin{defn}
The annihilator of a skew left brace $(A, \cdot, \circ)$ is defined as
	$$\Ann(A) := \Ker(\lambda) \cap \Z(A,\cdot) \cap \Fix(\lambda)=\Ker(\lambda) \cap \Z(A,\cdot) \cap \Z(A, \circ),$$
	where $\Ker(\lambda)$ denotes the kernel of $\lambda$ and $\Fix(\lambda) = \{ x \in A \mid \lambda_a(x) = x \text{ for all } a \in A \}$
\end{defn}

Note that, $\Ann(A)$ is an ideal of $(A, \cdot, \circ)$.

\begin{defn}
	The commutator $A'$ of a skew left brace $(A, \cdot, \circ)$ is the subgroup of $(A, \cdot)$ generated by elements of the form $a \cdot b\cdot a^{-1} \cdot b^{-1}$ and $a \cdot \lambda_b(a^{-1})$ for all $a, b \in A$. The commutator $A'$ turns out to be an ideal of $(A, \cdot, \circ)$.
\end{defn}

\begin{no} 
Before proceeding further, we set some notations:
\begin{itemize}
\item When the context is clear, we denote a skew left brace $(A, \cdot, \circ)$ by $A$ and its opposite skew left brace by $A^{\op}$. 
\item The product $a \cdot b$ in $(A, \cdot, \circ)$ is denoted by $ab$.
\item If $\lambda$ is the lambda map of  $A$ and  $I$ an ideal of $A$, then the lambda map of $A/I$ is denoted by $\overline{\lambda}$.
		\item If $G$ and $H$ are groups and $\psi: G \to \Aut(H)$ a homomorphism, then the group operation in the semi-direct product $H \rtimes_{\psi} G$ is given by 
		$$(h_1, g_1)(h_2, g_2)=(h_1 \psi_{g_1}(h_2), g_1 g_2)$$
		for all $h_1, h_2 \in H$ and $g_1, g_2 \in G$. 
		\item If $A$ is a skew left brace, then we denote by 
		$$\Lambda_{A} = (A,\cdot)\rtimes_{\lambda}(A,\circ) \quad \textrm{and} \quad \Lambda_{A^{\op}} = (A,\cdot)\rtimes_{\lambda^{\op}}(A,\circ)$$ the semi-direct products induced by $\lambda$ and $\lambda^{\op}$, respectively.
\item If $G$ is a group and $g \in G$, then $\Cl^G(g)$ denote the conjugacy class of $g$ in $G$ and $k(G)$ denote the number of conjugacy classes of $G$.
\end{itemize}
\end{no}
\medskip

\section{Some properties of $\Lambda_{A}$ and $\Lambda_{A^{\op}}$} \label{sec:properties}

In this section, we explore some basic structural properties of $\Lambda_{A}$ and $\Lambda_{A^{\op}}$. Direct checks lead to the following results.

\begin{prop}\label{comm1}
Let $A=(A, \cdot, \circ)$ be a skew left brace. Then the following assertions hold:
\begin{enumerate}
\item	$A'$ is an ideal of $A$ and $(A, \circ)'$ is contained in $A'$.
\item If $I$ is an ideal of $A$ containing $A'$ and $J$ is a normal subgroup  of $(A, \circ)$, then $I \rtimes_{\lambda^{\op}} J$ is a normal subgroup of  $\Lambda_{A^{\op}}$.
\item The annihilators of $(A, \cdot, \circ)$ and $(A,  \cdot^{\op}, \circ)$ are isomorphic.
\item The commutators of $(A, \cdot, \circ)$ and $(A, \cdot^{\op}, \circ)$ are the same as sets.
\end{enumerate}
\end{prop}

\begin{remark}
In general, the commutators of $(A, \cdot, \circ)$ and $(A, \cdot^{\op}, \circ)$ need not be isomorphic as skew left braces.  For example, if $(A, \cdot, \circ)$ is the trivial skew left brace over the symmetric group $S_n$ for $n \ge 5$, then the skew left braces $(A, \cdot, \circ)$ and $(A, \cdot^{\op}, \circ)$ are not isomorphic. Since their commutators, $A'$  and $(A^{\op})'$, equal the alternating group $A_n$ as a set, it follows that they are not isomorphic as skew left braces. 
\end{remark}

\begin{prop} \label{prop:AcongAop}
If $A$ is any skew left brace, then $\Lambda_{A} \cong \Lambda_{A^{\op}}$.
\end{prop}

\begin{proof}
Consider the map $\psi: \Lambda_{A^{\op}} \to \Lambda_{A}$ given by $\psi((a,b)) = (ab, b)$. Clearly, $\psi$ is a bijection. For $a_1, a_2, b_1, b_2 \in A$, we have
\begin{eqnarray*}
\psi((a_1, b_1)(a_2, b_2)) &=& \psi(a_1 \lambda^{\op}_{b_1}(a_2), b_1 \circ b_2)\\
 &=& \psi(a_1 \lambda^{\op}_{b_1}(a_2), b_1 \circ b_2)\\
 &=& (a_1 \lambda^{\op}_{b_1}(a_2) (b_1 \circ b_2), b_1 \circ b_2)\\
  &=& (a_1 b_1 \lambda_{b_1}(a_2) b_1^{-1} (b_1 \circ b_2), b_1 \circ b_2)\\
  &=& (a_1 b_1 \lambda_{b_1}(a_2)  \lambda_{b_1}(b_2), b_1 \circ b_2)\\
&=& (a_1 b_1 \lambda_{b_1}(a_2 b_2), b_1 \circ b_2)\\
&=&(a_1b_1, b_1) (a_2 b_2, b_2)\\
&=&\psi((a_1, b_1)) \psi((a_2, b_2)),
\end{eqnarray*}
and hence $\psi$ is an isomorphism of groups.
\end{proof}

Next, we compute the center of $\Lambda_{A^{\op}}$. Let $(x,y) \in \Z(\Lambda_{A^{\op}})$ and $(a, b) \in \Lambda_{A^{\op}}$. Then, we have
\begin{align}
	(x,y)(a,b)=&\;(a,b)(x,y)\nonumber\\
\Longleftrightarrow	(x \lambda^{\op}_{y}(a), y \circ b)= & \; (a  \lambda^{\op}_{b}(x), b \circ y)\nonumber\\
\Longleftrightarrow	(xy \lambda_y(a) y^{-1}, y \circ b)=& \;(ab \lambda_b(x) b^{-1}, b \circ y).\nonumber
\end{align}

Thus, $(x,y) \in \Z(\Lambda_{A^{\op}})$ if and only if $y \in \Z(A, \circ)$ and 
\begin{align} \label{cener1}
	xy \lambda_y(a) y^{-1} =& \; ab \lambda_b(x) b^{-1}
\end{align}
for all $a, b \in A$.

\begin{prop}\label{centersdp}
Let $A$ be any skew left brace. Then the following assertions hold:
\begin{enumerate}
\item $\Z(\Lambda_{A^{\op}}) \trianglelefteq \Fix(\lambda^{\op}) \rtimes_{\lambda^{\op}} \Z(A, \circ) \cong \Fix(\lambda^{\op}) \times \Z(A, \circ) $.
\item If $(x, y) \in \Z(\Lambda_{A^{\op}})$, then $\lambda_y(a)=y^{-1} x^{-1} a xy$ and $\lambda_a(y)= a^{-1} x^{-1} a xy$ for all $a \in A$.
\item The element $(x,1) \in \Z(\Lambda_{A^{\op}})$ if and only if $x \in \Z(A, \cdot) \cap \Fix(\lambda).$
\item The element  $(1,y) \in \Z(\Lambda_{A^{\op}})$ if and only if $y \in \Z(A, \circ) \cap \Fix(\lambda) \cap \Ker (\lambda^{\op})$.
\item $\Ann(A) \times \Ann(A) \trianglelefteq \Z(\Lambda_{A^{\op}})$.
\end{enumerate}
\end{prop}

\begin{proof}
The proofs are straightforward.
	\begin{enumerate}
		\item If $(x, y) \in \Z(\Lambda_{A^{\op}})$, then putting $a=1$ in \eqref{cener1} gives $\lambda_b(x) = b^{-1}xb$ for all $b \in A$.
		\item Putting $b=1$ in \eqref{cener1} gives $\lambda_y(a) = y^{-1} x^{-1} a xy$ for all $a \in A$. Further, using $y \lambda_y(a) = a \lambda_a(y)$ gives $\lambda_a(y)= a^{-1} x^{-1} a xy$ for all $a \in A$.
		\item This follows by putting $y=b=1$ in \eqref{cener1}.
		\item This follows by putting $x=1$ in $(2)$ and $(3)$.
		\item It can be easily seen that \eqref{cener1} holds for all $x, y \in \Ann(A)$.
	\end{enumerate}
\end{proof}

\begin{cor}
If $A$ is a left brace, then 
	\[ \Z(\Lambda_{A^{\op}}) = \Fix(\lambda) \rtimes_{\lambda^{\op}} \Ann(A) \cong \Fix(\lambda) \times \Ann(A). \]
\end{cor}

\begin{proof}
Since $A$ is a left brace, by \eqref{cener1}, we have that $(x, y) \in \Z(\Lambda_{A^{\op}})$ if and only if
	\begin{align*} 
		x \lambda_y(a) = a \lambda_b(x)
	\end{align*}
for all $a, b \in A$. Let $(x, y) \in \Fix(\lambda) \times \Ann(A)$ and $a, b \in A$. Then we have
$$
a \lambda_b(x) = ax \quad \text{and} \quad 	x \lambda_y(a) = xa. 
$$
This shows that $\Fix(\lambda) \times \Ann(A) \leq \Z(\Lambda_{A^{\op}})$. Conversely, let $(x, y) \in \Z(\Lambda_{A^{\op}})$. By Proposition \ref{centersdp}(1) and the fact that $y \in \Z(A, \circ)$, it follows that $y \in \Ann(A)$, and hence $(x, y) \in \Fix(\lambda) \times \Ann(A)$. 
\end{proof}

\begin{remark}
	It is easy to see that $\Lambda_{A^{\op}}/ (\Ann(A) \times \Ann(A)) \cong A/\Ann(A) \rtimes_{\overline{\lambda^{\op}}} A/\Ann(A)$ under the natural map $\overline{(a,b)} \mapsto (\overline{a}, \overline{b})$ for all $a, b \in A$.	The group $ A/\Ann(A) \rtimes_{\overline{\lambda^{\op}}} A/\Ann(A)$ will be denoted by $\overline{\Lambda}_{A^{\op}}$.
\end{remark}

\begin{thm}\label{commutator1}
If $\Lambda_{A^{\op}}'$ is the commutator subgroup of $\Lambda_{A^{\op}}$, then  $\Lambda_{A^{\op}}'=A' \rtimes_{\lambda^{\op}} (A, \circ)'$.  
\end{thm}
\begin{proof}
It follows from Proposition \ref{comm1}(2) that \( A' \rtimes_{\lambda^{\op}} (A, \circ)' \) is a normal subgroup of \( \Lambda_{A^{\op}} \).  The inverse of $(a, b) \in\Lambda_{A^{\op}}$ is $(\lambda^{\op}_{b^\dagger}(a^{-1}), b^\dagger)$. For $a_1, a_2, b_1, b_2 \in A$, we have
	\begin{align*}
		(a_1, b_1)(a_2, b_2)(a_1, b_1)^{-1}(a_2, b_2)=\; &(a_1 \lambda^{\op}_{b_1}(a_2), b_1 \circ b_2)(\lambda^{\op}_{b^\dagger_1}(a_{1}^{-1}), b^\dagger_1) (\lambda^{\op}_{b^\dagger_2}(a_2^{-1}), b^\dagger_2)\\
		=\; &(a_1 \lambda^{\op}_{b_1}(a_2), b_1 \circ b_2)(\lambda^{\op}_{b^\dagger_1}(a_{1}^{-1}) \lambda^{\op}_{b^\dagger_1 \circ b^\dagger_2}(a_2^{-1}), b^\dagger_1 \circ b^\dagger_2 )\\
		=\;& (a_1 \lambda^{\op}_{b_1}(a_2) \lambda^{\op}_{b_1 \circ b_2 \circ b^\dagger_1}(a_{1}^{-1}) \lambda^{\op}_{b_1 \circ b_2 \circ b^\dagger_1 \circ b^\dagger_2}(a_2^{-1}), b_1 \circ b_2 \circ b^\dagger_1 \circ b^\dagger_2).
	\end{align*}
	
	We can write 
	\begin{align*}
		& a_1  \lambda^{\op}_{b_1}(a_2) \lambda^{\op}_{b_1 \circ b_2 \circ b^\dagger_1}(a_{1}^{-1}) \lambda^{\op}_{b_1 \circ b_2 \circ b^\dagger_1 \circ b^\dagger_2}(a_2^{-1})\\
		&=\;\big(a_1 \lambda^{\op}_{b_1 \circ b_2 \circ b^\dagger_1}(a_{1}^{-1}) \big) \big(\lambda^{\op}_{b_1 \circ b_2 \circ b^\dagger_1}(a_{1}) \lambda^{\op}_{b_1}(a_2) \lambda^{\op}_{b_1 \circ b_2 \circ b^\dagger_1}(a_{1}^{-1}) \lambda^{\op}_{b_1}(a_2^{-1}) \big)\big(\lambda^{\op}_{b_1}(a_2) \lambda^{\op}_{b_1 \circ b_2 \circ b^\dagger_1 \circ b^\dagger_2}(a_2^{-1}) \big).
	\end{align*}
	This shows that $a_1  \lambda^{\op}_{b_1}(a_2) \lambda^{\op}_{b_1 \circ b_2 \circ b^\dagger_1}(a_{1}^{-1}) \lambda^{\op}_{b_1 \circ b_2 \circ b^\dagger_1 \circ b^\dagger_2}(a_2^{-1}) \in (A^{\op})^\prime$. But, since $(A^{\op})^\prime=A^\prime$, we obtain $\Lambda^{\prime}_{A^{\op}} \leq A' \rtimes_{\lambda^{\op}} (A, \circ)'$.
\par

For the converse, since
	$$(a,b)(1, b)(a,b)^{-1}(1, b)^{-1} \in \Lambda'_{A^{\op}},$$ we obtain $((a\lambda^{\op}_{b}(a^{-1}), 1) \in \Lambda'_{A^{\op}}$ for all $a, b \in A$. Similarly, since  $$(a_1,1 )(a_2,1)(a_1,1)^{-1}(a_2, 1)^{-1} \in \Lambda'_{A^{\op}},$$ we obtain $(A, \cdot)' \times \{1\} \leq \Lambda'_{A^{\op}}$. From this we conclude that $A' \times \{1\} \le \Lambda'_{A^{\op}}$. It is easy to see that $\{1\} \times (A, \circ)' \leq \Lambda'_{A^{\op}}.$ This implies that $A' \rtimes_{\lambda^{\op}}(A, \circ)' \leq \Lambda'_{A^{\op}}$, and hence $A' \rtimes_{\lambda^{\op}}(A, \circ)' =\Lambda'_{A^{\op}}$.
\end{proof}

\begin{cor} \label{cor:LamAmodLamAprime}
	The group $\Lambda_{A^{\op}}/ \Lambda^{\prime}_{A^{\op}}$ is isomorphic to $A/A^\prime \times (A, \circ)/ (A, \circ)^\prime$.
\end{cor}

\begin{proof}
	It is easy to see that the map $\Lambda_{A^{\op}} / \Lambda'_{A^{\op}}  \to A/A' \times (A, \circ) / (A, \circ)'$ given by
$\overline{(a,b)}  \mapsto (\overline{a}, \overline{b})$ is an isomorphism of groups.
 \end{proof}

The notion of isoclinism of groups was introduced by Hall in \cite{MR0003389}  to facilitate the classification of prime power groups.

\begin{defn}
Two groups $G$ and $H$ are said to be isoclinic if there exist group isomorphisms $\xi_1: G/ \Z(G) \to H/ \Z(H)$ and $\xi_2: G' \to H'$ such that the diagram
	\begin{align}\label{cdsb2}
		\begin{CD}
\big(G/ \Z(G)\big) \times \big(G/ \Z(G)\big)@>{\eta_G}>> G' \\
@V{\xi_1 \times \xi_1}VV @V{\xi_2}VV\\
\big(H/ \Z(H)\big) \times \big(H/ \Z(H)\big)@>{\eta_H}>> H'
		\end{CD}	
	\end{align}
 commutes. Here, the maps $\eta_G: (G/ \Z(G)) \times (G/ \Z(G)) \rightarrow G' $ and $\eta_H: (H/ \Z(H)) \times (H/ \Z(H)) \rightarrow H' $ are defined by 
	$$\eta_G \big(\overline{g_1}, \overline{g_2} \big)=g_1g_2g_1^{-1}g_2^{-1} \text{ and }  \eta_H \big(\overline{h_1}, \overline{h_1} \big)=h_1h_2h_1^{-1}h_2^{-1}$$
for all $g_1, g_1 \in G$ and $h_1, h_2 \in H$.
\end{defn}

 We need the following result of Hall \cite{MR0003389}.

\begin{thm}\label{gpiso}
Let $G$ be a group and $N \trianglelefteq G$. Then $G$ is isoclinic to $G/N$ if and only if $N \cap G'=\{1\}$.
\end{thm}

In \cite{LV24}, Letourmy and Vendramin introduced isoclinism of skew left braces and investigated properties  that are invariant under
this equivalence.

\begin{defn}
Two skew left braces $A$ and $B$ are said to be isoclinic if there exist skew brace isomorphisms $\xi_1: A/ \Ann(A) \to B/ \Ann(B)$ and $\xi_2: A' \to B'$ such that the diagram
	\begin{align}\label{cdsb}
		\begin{CD}
			A'  @<{\theta_A}<<\big(A/ \Ann(A)\big) \times \big(A/ \Ann(A)\big)@>{\theta^{*}_A}>> A' \\
			@V{\xi_2}VV @V{\xi_1 \times \xi_1}VV @V{\xi_2}VV\\
			B' @<{\theta_B}<< \big(B/ \Ann(B)\big) \times \big(B/ \Ann(B)\big)@>{\theta^{*}_B}>> B'
		\end{CD}	
	\end{align}
 commutes. Here, the maps $\theta_A, \theta^*_A: (A/ \Ann(A)) \times (A/ \Ann(A)) \rightarrow A' $ are defined by 
	$$\theta_A(\overline{a}, \overline{b})=aba^{-1}b^{-1} \text{ and }  \theta^*_A(\overline{a}, \overline{b})=\lambda_a(b)b^{-1}$$
for all $a,b \in A.$ The maps $\theta_B$ and  $\theta^*_B$ are defined similarly.
\end{defn}

The following result is due to \cite[Proposition 3.10]{LV23}.

\begin{prop}
	If $A$ and $B$ are isoclinic skew left braces, then $(A, \cdot)$ is isoclinic to $(B, \cdot)$ and $(A, \circ)$ is isoclinic to $(B,\circ)$ as groups.
\end{prop}

\begin{thm}\label{isocom}
If $A$ and $B$ are  isoclinic skew left braces, then $\Lambda'_{A^{\op}} \cong \Lambda'_{B^{\op}}$.
\end{thm}
\begin{proof}
Suppose that $A$ and $B$ are isoclinic skew left braces. Then we have a skew left brace isomorphism $\xi_2: A' \rightarrow B'$ such that $\xi_2|: (A, \circ)' \rightarrow (B, \circ)'$ is an isomorphism of groups, where $\xi_2|$ denotes the restriction of $\xi_2$ to $(A, \circ)'$. By Theorem \ref{commutator1}, we have $\Lambda'_{A^{\op}}=A' \rtimes_{\lambda^{\op}} (A, \circ)' $ and $\Lambda'_{B^{\op}}=B' \rtimes_{\lambda^{\op}} (B, \circ)' $. Define $\phi: \Lambda'_{A^{\op}} \rightarrow \Lambda'_{B^{\op}}$ by $\phi(a,b) = (\xi_2(a), \xi_2|(b))$ for all $a \in A'$ and $b \in (A, \circ)'$. Since $\xi_2$ is an isomorphism of skew left braces, it follows that $\phi$ is an isomorphism of groups, which is desired.
\end{proof}

\begin{thm}\label{isoclinicsdp}
	Let $A$ and $B$ be isoclinic skew left braces. Then $\Lambda_{A^{\op}}$ and $\Lambda_{B^{\op}}$ are isoclinic groups if any of the following conditions are satisfied:
	\begin{enumerate}
		\item $\Z(\Lambda_{A^{\op}}) = \Ann(A) \times \Ann(A)$ and $\Z(\Lambda_{B^{\op}}) = \Ann(B) \times \Ann(B)$,
		
		\item $\Ann(A) \cap A'=\{1\}$ and $\Ann(B) \cap B'=\{1\}$.
	\end{enumerate}
\end{thm}
\begin{proof} 
	\begin{enumerate}
		\item In this case, we have $\Lambda_{A^{\op}}/ \Z(\Lambda_{A^{\op}}) \cong \overline{\Lambda}_{A^{\op}}$ and $\Lambda_{B^{\op}}/ \Z(\Lambda_{B^{\op}}) \cong \overline{\Lambda}_{B^{\op}}$. Since $A$ and $B$ are isoclinic skew left braces, there exist isomorphisms of skew left braces $\xi_1: A/ \Ann(A) \to B/ \Ann(B)$ and $\xi_2: A' \to B'$ such that the diagram \eqref{cdsb} commutes.
Since $\xi_1$ is an isomorphism of skew left braces, it follows that the map $\overline{\xi}_1: \overline{\Lambda}_{A^{\op}} \rightarrow \overline{\Lambda}_{B^{\op}}$ given by $\overline{\xi}_1(\overline{a}, \overline{b}) = (\xi_1(\overline{a}), \xi_1(\overline{b}))$ is an isomorphism of groups. Further, $\xi_2$ being an isomorphism of skew left braces implies that $\xi_2 \times \xi_2|: \Lambda_{A^{\op}}' \to \Lambda_{B^{\op}}'$ is an isomorphism of groups. The commutativity of the diagram
		$$\begin{CD}\label{cdgp}
			(\overline{\Lambda}_{A^{\op}}) \times (\overline{\Lambda}_{A^{\op}}) @>\eta_{\Lambda_{A^{\op}}}>> \Lambda_{A^{\op}}'\\
			@V{\overline{\xi}_1 \times \overline{\xi}_1}VV @V{\xi_2 \times \xi_2|}VV\\
			(\overline{\Lambda}_{B^{\op}}) \times (\overline{\Lambda}_{B^{\op}}) @>\eta_{\Lambda_{B^{\op}}}>> \Lambda^\prime_{B^{\op}}
		\end{CD}$$
follows from the commutativity of the diagram \eqref{cdsb}, which shows that $\Lambda_{A^{\op}}$ and $\Lambda_{B^{\op}}$ are isoclinic groups.
		
		\item It is easy to see that $\Ann(A) \cap A' =\{1\}$ and  $\Ann(B) \cap B' = \{1\}$ implies that $\Lambda'_{A^{\op}} \cap (\Ann(A) \times \Ann(A)) =\{1\}$ and $\Lambda'_{B^{\op}} \cap (\Ann(B) \times \Ann(B)) = \{1\}$. Hence, by Theorem \ref{gpiso}, we have that \( \Lambda_{A^{\op}} \) is isoclinic to \( \overline{\Lambda}_{A^{\op}} \) and \( \Lambda_{B^{\op}} \) is isoclinic to $\overline{\Lambda}_{B^{\op}} $.  Since $A$ and $B$ are isoclinic skew left braces, there exist an isomorphism $\xi_1: A/ \Ann(A) \to B/ \Ann(B)$  of skew left braces. 
This implies that $\overline{\Lambda}_{A^{\op}}$ is isomorphic to $\overline{\Lambda}_{B^{\op}}$, and hence $\Lambda_{A^{\op}} \) is isoclinic to \( \Lambda_{B^{\op}}$.
	\end{enumerate}
\end{proof}

\begin{remark}
Let $A$ be a trivial skew left brace with a non-abelian simple additive group, and let $B = A^{\op}$ its opposite skew left brace. It is easy to see that $A$ and $B$ are non-isoclinic skew left braces. However, $\Lambda_{A^{\op}} \cong \Lambda_{B^{\op}} \cong A \times A$ as groups. Thus, $\Lambda_{A^{\op}}$ and $\Lambda_{B^{\op}}$ are isoclinic as groups, but $A$ and $B$ are not isoclinic as skew left braces.
\end{remark}
\medskip

\section{Representations of skew left braces} \label{sec:Rep}
In  \cite{LV23}, Letourmy and Vendramin defined a representation of a skew left brace as follows.
		\begin{defn} A representation of a skew left brace $A$ is a triple $(V,\beta,\rho)$, where
		\begin{enumerate}[(1)]
			\item $V$ is a vector space over some field;
			\item $\beta : (A,\cdot)\longrightarrow \GL(V)$ is a representation of $(A,\cdot)$;
			\item $\rho: (A,\circ)\longrightarrow \GL(V)$ is a representation of $(A,\circ)$;
		\end{enumerate}
		such that the relation
		\begin{equation}\label{relation}
			\beta \big(\lambda^{\op}_a(b)\big) = \rho(a)\beta(b)\rho(a)^{-1}
		\end{equation}
		holds for all $a,b\in A$. 
	\end{defn}

\begin{prop} \label{prop:Rep1}
Let $A$ be a skew left brace and $\beta: A \rightarrow \GL(V)$ be a skew brace homomorphism considering $\GL(V)$ to be a trivial skew left brace. Then $(V,\beta,\beta)$ is a representation of $A$.
\end{prop}
\begin{proof}
	Since $\beta$ is a skew brace homomorphism, we have $\beta(a\cdot b) = \beta(a) \beta(b)$ and $\beta(a\circ b) = \beta(a) \beta(b)$ for all $a, b \in A$. Then $$\beta\left( \lambda^{\op}_{a}(b) \right) = \beta((a\circ b) \cdot a^{-1}) = \beta(a\circ b)\beta(a)^{-1} = \beta(a)\beta(b)\beta(a)^{-1}$$
	for all $a, b \in A$, and hence $(V, \beta, \beta)$ is a representation of $A$.
\end{proof}

\begin{example} \label{prop:Rep2}
	Let $X$ be a set and $V$ be a vector space over a field with $X$ as its basis. Let $A$ be a skew left brace and $\eta: A \rightarrow \Sym(X)$ be a skew brace homomorphism considering $\Sym(X)$ to be a trivial skew brace. Then the map $\beta_{\eta}: A \rightarrow \GL(V)$ given by $\beta_{\eta}(a)=\eta(a)$ is a skew brace homomorphism considering $\GL(V)$ to be a trivial skew brace. Hence, $(V, \beta_{\eta}, \beta_{\eta})$ is a representation of $A$.
\end{example}

Kozakai and Tsang \cite[Remark 2.2]{KT24} established a one-way correspondence between representations of a skew left brace $A$ and that of the group $\Lambda_{A^{\op}}$ in the following way. 

	\begin{lemma} \label{lemma:repbracetogroup}
Let  $(V,\beta,\rho)$ be a representation of a skew left brace $A$. Then there is a representation $(V,\varphi_{(\beta,\rho)})$ of the group $\Lambda_{A^{\op}}$, where  $\varphi_{(\beta,\rho)} : \Lambda_{A^{\op}} \longrightarrow \GL(V)$ is given by $$\varphi_{(\beta,\rho)}(a,b) = \beta(a)\rho(b)$$
for all $a, b \in A$.
\end{lemma}

In fact, we prove that the converse of Lemma \ref{lemma:repbracetogroup} also holds.

\begin{thm} \label{thm:grouptoskewbrace}
Let $A$ be a skew left brace and  $(V, \phi)$ be a representation of the group $\Lambda_{A^{\op}}$. Then $(V,\beta_{\phi},\rho_{\phi})$ forms a representation of $A$, where 	$\beta_\phi: (A, \cdot) \rightarrow \GL(V)$ and 	$\rho_\phi: (A, \circ) \rightarrow \GL(V)$ are given by 
	$\beta_{\phi}(a):=\phi(a,1)$ and $\rho_{\phi}(b):=\phi(1,b)$ for all $a,b \in A$.
\end{thm}

\begin{proof}
For $a_1, a_2, b_1, b_2\in A$, we have
$$\beta_\phi(a_1) \beta_\phi(a_2) = \phi(a_1, 1)\phi(a_2, 1) = \phi((a_1, 1)(a_2, 1)) = \phi\left( a_1\lambda^{\op}_{1}(a_2), 1 \right) = \phi(a_1 a_2, 1) = \beta_\phi(a_1 a_2)$$
	and 
$$\rho_\phi(b_1) \rho_\phi(b_2) =  \phi(1, b_1)\phi(1, b_2) = \phi((1, b_1)(1, b_2)) = \phi\left( \lambda^{\op}_{b_1}(1), b_1 b_2 \right) = \phi(1, b_1 b_2) = \rho_\phi(b_1 b_2).$$
	Thus, $\beta_\phi$ and $\rho_\phi$ are group homomorphisms. Further, for $a, b\in A$, we have
	\begin{align*}
		\rho_\phi(b)\beta_\phi(a)\rho_\phi(b)^{-1} & = \phi(1, b) \phi(a, 1) \phi(1, b^{-1}) = \phi\left( \lambda^{\op}_{b}(a), b \right) \phi(1, b^{-1}) = \phi\left( \lambda^{\op}_{b}(a) \lambda^{\op}_{b}(1), 1 \right) \\
		& = \phi\left( \lambda^{\op}_{b}(a), 1 \right) \\
		& = \beta_\phi\left( \lambda^{\op}_{b}(a) \right),
	\end{align*}
and hence $(V,\beta_{\phi},\rho_{\phi})$ forms a representation of $A$. Finally, by Lemma \ref{lemma:repbracetogroup}, there is a one-to-one correspondence between the representations of the group $\Lambda_{A^{\op}}$ and the representations of the skew left brace $A$.
\end{proof}

\begin{defn}
A representation $(V,\beta,\rho)$  of a skew left brace $A$ is called irreducible if the induced representation $(V,\varphi_{(\beta,\rho)})$ of the group $\Lambda_{A^{\op}}$ is irreducible. Further, two representations $(V_1,\beta_1,\rho_1)$ and $(V_2,\beta_2,\rho_2)$ of $A$ are called 
equivalent if the representations $(V_1,\varphi_{(\beta_1,\rho_1)})$ and $(V_2,\varphi_{(\beta_2,\rho_2)})$ of $\Lambda_{A^{\op}}$ are
equivalent.
\end{defn}

\begin{cor}\label{one one corres rep}
There is a one-to-one correspondence between the set of equivalence classes of irreducible representations of a skew left brace $A$ and the set of equivalence classes of irreducible representations of the group $\Lambda_{A^{\op}}$.
\end{cor}

A subspace $W$ of a vector space $V$ is called proper if $\{ 0\} \neq W \neq V$.

\begin{prop} \label{prop:invariantsubspaces}
Let $A$ be a skew left brace and $(V, \beta, \rho)$ be a representation of $A$. Then $(V, \beta, \rho)$ is irreducible if and only if no proper subspace of $V$ is invariant under both $\beta$ and $\rho$.

\end{prop}
\begin{proof}
Suppose that $\phi_{(\beta, \rho)}$ is an irreducible representation of $\Lambda_{A^{\op}}$. If $W$ is a subspace of $V$ which is invariant under both  $\beta$ and $\rho$, then we see that 
$$\phi_{(\beta, \rho)}(a,b) (W) = \beta(a)(\rho(b)(W)) \subseteq \beta(a)(W) \subseteq W$$
for all $(a,b) \in \Lambda_{A^{\op}}$. This implies that $W$ cannot be a proper subspace of $V$.
\par 
Conversely, suppose that no proper subspace of $V$ is invariant under both  $\beta$ and $\rho$. Let $W$ be a subspace of $V$ which is invariant under $\phi_{(\beta, \rho)}$, that is, $\phi_{(\beta, \rho)}(a, b)(W) \subseteq W$ for all $(a,b) \in \Lambda_{A^{\op}}$. This implies that $\beta(a)(W) \subseteq W$ for all $a\in (A, \cdot)$ and  $\rho(a)(W) \subseteq W$ for all $a\in (A, \circ)$. Hence, $W$ is not a proper subspace of $V$, and therefore $\phi_{(\beta, \rho)}$ is irreducible.
\end{proof}

In view of Corollary \ref{one one corres rep}, it is sufficient to look at the representations of $\Lambda_{A^{\op}}$ to understand the  representations of $A$. Henceforth, we assume that all our representations are over the field $\mathbb{C}$ of complex numbers. Let $\ird(A)$ denote the set of dimensions of all irreducible representations of the skew left brace $A$. In the following examples, we obtain $\ird(A)$ for some classes of skew left braces. 

\begin{example} \label{ex:repbracep^2}
	\textnormal{
		Let $A$ be a left brace of order $p^2$, where $p$ is a prime. By Theorem \ref{commutator1}, we have $\Lambda'_{A^{\op}}=A' \rtimes (A, \circ)'$. Since $(A, \circ)$ is abelian, and  $|A^{\prime}| \leq p$, we get $|\Lambda'_{A^{\op}}| = 1$ or $p$.
		\begin{itemize}
			\item If $|\Lambda'_{A^{\op}}| = 1$, then $\Lambda_{A^{\op}}$ is an abelian group of order $p^4$. By \cite[Corollary 2.7]{IBook}, $\Lambda_{A^{\op}}$ has $p^4$ many inequivalent representations of dimension 1.
			\item If  $|\Lambda'_{A^{\op}}| = p$, then $\Lambda_{A^{\op}}$ is a non-abelian group of order $p^4$, and $|\Z(\Lambda_{A^{\op}})| = p$ or $p^2$. In this case, $|\Lambda_{A^{\op}}/ \Lambda'_{A^{\op}}| = p^3$. Hence, by \cite[Corollary 2.23]{IBook}, $\Lambda_{A^{\op}}$ has $p^3$ many inequivalent representations of dimension 1. Since $\Lambda_{A^{\op}}$ is a non-abelian group, let $\alpha$ be a non-linear (degree more than 1) irreducible representation of $\Lambda_{A^{\op}}$ and let $\chi$ be the character of $\alpha$. By \cite[Corollary 2.30]{IBook}, $\chi(1)^2 \leq |\Lambda_{A^{\op}}/\Z(\Lambda_{A^{\op}})| \leq p^3$. Thus, $\dim(\alpha) = \chi(1) = p$, and hence $\ird(\Lambda_{A^{\op}}) = \{ 1, p \}$. By \cite[Corollary 2.7]{IBook}, we have
$$p^4 = |\Lambda_{A^{\op}}| = p^3 \cdot 1 + r \cdot p^2,$$
			where $r$ is the number of inequivalent irreducible representations of dimension $p$. Thus, $\Lambda_{A^{\op}}$ has 
$r = p^2 - p$ many inequivalent irreducible representations of dimension $p$. 
\end{itemize}}
\end{example}

\begin{example} \label{example:skewbracep^3}
\textnormal{
Let $A$ be a skew brace of order $p^3$, where $p>2$ is a prime, such that $|\Z(\Lambda_{A^{\op}})|\geq p^3$. Note that, there exist groups of order $p^6$ with center of order $p^3$, such as the groups belonging to the isoclinic family $\Phi_{i}$ for $i\in \{ 2,3,4,6,11 \}$ as listed in \cite[Table 1]{NO'BV23}. By Theorem \ref{commutator1}, we have $\Lambda'_{A^{\op}}=A' \rtimes (A, \circ)'$. Since $(A, \circ)$ is either abelian, or an extra-special group of order $p^3$, and $|A'| \leq p^2$, we get  $|\Lambda'_{A^{\op}}| = 1$, $p$, $p^2$, or $p^3$.
		\begin{itemize}
			\item If $|\Lambda'_{A^{\op}}| = 1$, then $\Lambda_{A^{\op}}$ is an abelian group of order $p^4$. By \cite[Corollary 2.7]{IBook}, $\Lambda_{A^{\op}}$ has $p^4$ many inequivalent representations of dimension 1.
			\item If  $|\Lambda'_{A^{\op}}| = p$, then $\Lambda_{A^{\op}}$ is a non-abelian group of order $p^6$ belonging to the isoclinic family $\Phi_{2}$ of \cite[Table 1]{NO'BV23}. In this case, $|\Z(\Lambda_{A^{\op}})| = p^4$, and hence $\ird(\Lambda_{A^{\op}}) = \{ 1, p \}$. By \cite[Corollary 2.7]{IBook}, we have
			\[ p^6 = |\Lambda_{A^{\op}}| = |\Lambda_{A^{\op}}/\Lambda'_{A^{\op}}| + r\cdot p^2, \]
			where $r$ is the number of inequivalent irreducible representations of dimension $p$. Thus, $\Lambda_{A^{\op}}$ has 
$r = p^4 - p^3$ many inequivalent irreducible representations of dimension $p$.
			\item If  $|\Lambda'_{A^{\op}}| = p^2$, then $\Lambda_{A^{\op}}$ has $|\Lambda_{A^{\op}}/\Lambda'_{A^{\op}}| = p^4$ many inequivalent representations of dimension 1. Since $|\Z(\Lambda_{A^{\op}})|\geq p^3$, we have $|\Lambda_{A^{\op}}/\Z(\Lambda_{A^{\op}})|\leq p^3$. By \cite[Corollary 2.30]{IBook}, we have $\ird(\Lambda_{A^{\op}}) = \{ 1, p \}$. Then 
			\[ p^6 = |\Lambda_{A^{\op}}| = p^4 + r\cdot p^2, \]
			where $r$ is the number of inequivalent irreducible representations of dimension $p$. Hence, $\Lambda_{A^{\op}}$ has 
$r = p^4 - p^2$ many inequivalent irreducible representations of dimension $p$.
			\item If  $|\Lambda'_{A^{\op}}| = p^3$, then $\Lambda_{A^{\op}}$ has $|\Lambda_{A^{\op}}/\Lambda'_{A^{\op}}| = p^3$ many inequivalent representations of dimension 1. Since $|\Z(\Lambda_{A^{\op}})|\geq p^3$, we have $|\Lambda_{A^{\op}}/\Z(\Lambda_{A^{\op}})|\leq p^3$. By \cite[Corollary 2.30]{IBook}, we have $\ird(\Lambda_{A^{\op}}) = \{ 1, p \}$. Then 
			\[ p^6 = |\Lambda_{A^{\op}}| = p^3 + r\cdot p^2, \]
			where $r$ is the number of inequivalent irreducible representations of dimension $p$. Thus, $\Lambda_{A^{\op}}$ has 
$r = p^4 - p$ many inequivalent irreducible representations of dimension $p$.
		\end{itemize}
	}
\end{example}

\begin{example} \label{example:bicyclicskewbrace}
{\rm  A left brace is called bi-cyclic if both the additive and the multiplicative groups are cyclic \cite{WR08}. Let $A$ be a bi-cyclic left brace of order $p^n$, where $p>2$ is a prime and $n \ge 2$. Then $\Lambda_{A^{\op}}$ is a metacyclic $p$-group of the form $\langle a, b \mid a^{p^{n}} =1,  b^{p^{n}} = 1, b^{-1}ab = a^{1+p^{r}} \rangle$,
		where $0< r <n$. A direct check gives $\Lambda_{A^{\op}}' = \langle a^{p^r} \rangle$. Thus, $\Lambda_{A^{\op}}$ has $|\Lambda_{A^{\op}}/\Lambda_{A^{\op}}'| = p^{n+r}$ many inequivalent representations of  dimension 1.  Further, by \cite[Theorem 4.29]{B95}, we have 
$$\ird(\Lambda_{A^{\op}}) = \{ 1, \,p^{n-r-\alpha} \mid 0< r< n \text{ and } 0 \leq \alpha < n-r \}.$$}		
\end{example}

Next, we derive some results on representations using the structural properties proved in Section \ref{sec:properties}.

\begin{prop} \label{prop:irrAcircsubsetirrLambdaA}
If $A$ is a skew left brace, then
$$\ird(A, \circ) \subseteq \ird(\Lambda_{A^{\op}}).$$
\end{prop}
\begin{proof}
	The result follows from the fact that if $\rho: (A, \circ) \to \GL(V)$ is an irreducible representation, then $(V, \beta_{0}, \rho)$ is an irreducible representation of $A$, where  $\beta_0 : (A, \cdot) \to \GL(V)$ is the trivial representation.
\end{proof}

\begin{cor}
Let $A$ and $B$ be skew left braces such that $(A, \circ)$ and $(B, \circ)$ are isoclinic groups. Then
$$\ird(A, \circ)  = \ird(B, \circ) \subseteq \ird(\Lambda_{A^{\op}}) \cap \ird(\Lambda_{B^{\op}}).$$
\end{cor}
\begin{proof}
	Since $(A, \circ)$ and $(B, \circ)$ are isoclinic groups, by \cite[Theorem 3.2]{T76}, we have $\ird(A, \circ)  = \ird(B, \circ)$. The result now follows from Proposition \ref{prop:irrAcircsubsetirrLambdaA}.
\end{proof}

\begin{cor}
	Let $A$ and $B$ be isoclinic skew left braces such that any of the following conditions are satisfied:
	\begin{enumerate}
		\item $\Z(\Lambda_{A^{\op}}) = \Ann(A) \times \Ann(A)$ and $\Z(\Lambda_{B^{\op}}) = \Ann(B) \times \Ann(B)$.
		\item $\Ann(A) \cap A^\prime=\{ 1\}$ and $\Ann(B) \cap B^\prime=\{1\}$.
	\end{enumerate}
	Then $\ird(\Lambda_{A^{\op}}) = \ird(\Lambda_{B^{\op}})$.
\end{cor}
\begin{proof}
	The result follows from Theorem \ref{isoclinicsdp} and \cite[Theorem 3.2]{T76}.
\end{proof}

\begin{remark}
\textnormal{
If $A$ is a trivial skew left brace, then by Proposition \ref{prop:AcongAop}, $\Lambda_{A} \cong \Lambda_{A^{\op}}$. Hence, the representations of $A$ and $A^{\op}$ are in one-to-one correspondence. Note that, $A$ and $A^{\op}$ are non-isomorphic skew left braces when $(A, \cdot)$ is a non-abelian group.
} 
\end{remark}

\begin{cor}
Let $A$ be a skew left brace such that $\Ann(A)$ is non-trivial. Then $\Lambda_{A^{\op}}$ does not admit any faithful irreducible representation.
\end{cor}
\begin{proof}
By Proposition \ref{centersdp}(5), we have $\Ann(A) \times \Ann(A) \trianglelefteq \Z(\Lambda_{A^{\op}})$. Thus, if $\Ann(A)$ is non-trivial, then $\Z(\Lambda_{A^{\op}})$ is not cyclic. Then, it follows from \cite[Theorem 2.32(a)]{IBook} that $\Lambda_{A^{\op}}$ does not admit any faithful irreducible representation.
\end{proof}

\begin{cor}
	Let $A$ be a skew left brace. Then the number of 1-dimensional representations of  $A$ is 
	\[ |A/A'| \cdot \big(\text{The number of 1-dimensional representations  of } (A, \circ ) \big). \]
\end{cor}

\begin{proof}
By \cite[Corollary 2.23]{IBook}, the number of 1-dimensional representations of a group $G$ is $|G/G'|$. The result now follows from Corollary \ref{cor:LamAmodLamAprime}.
\end{proof}

\begin{cor}
Let $A$ and $B$ be isoclinic skew left braces of the same order. Then $A$ and $B$ admit the same number of 1-dimensional representations.
\end{cor}
\begin{proof}
The result follows from Theorem \ref{isocom} and \cite[Corollary 2.23]{IBook}.
\end{proof}

Let $\phi$, $\beta_{reg}$ and $\rho_{reg}$ be the regular representations of the groups $\Lambda_{A^{\op}}$, $(A, \cdot)$ and $(A, \circ)$, respectively. Let $\beta_{\phi}$ and $\rho_{\phi}$ be the representations of $(A, \cdot)$ and $(A, \circ)$, respectively, induced from $\phi$ (see Theorem \ref{thm:grouptoskewbrace}).

\begin{thm} \label{thm:regularrep}
	Let $A$ be a finite skew left brace and $\phi$ be the regular representation of $\Lambda_{A^{\op}}$. Then 
$$\beta_\phi = |A|\cdot \beta_{reg} \quad \text{and} \quad \rho_\phi = |A| \cdot \rho_{reg}.$$

\end{thm}
\begin{proof}
	Let $A = \{ a_1, a_2, \ldots, a_n = 1 \}$. Then $\beta_{reg}: (A, \cdot) \rightarrow \GL(kA)$ is given by $\beta_{reg}(a_s)(a_i) = a_s \cdot a_i = a_{k_{si}}$ for some $1\leq k_{si}\leq n$. Similarly, $\rho_{reg}: (A, \circ) \rightarrow \GL(kA)$ is given by $\rho_{reg}(a_s)(a_j) = a_s \circ a_j = a_{l_{sj}}$ for some $1\leq l_{sj}\leq n$.	Then, for each $1\leq s \leq n$, the matrices of $\beta_{reg}(a_s)$ and $\rho_{reg}(a_s)$ are given by
$$
	\left[ \beta_{reg}(a_s) \right] = 
	\begin{cases}
		1 & \quad \textrm{for}~(k_{si}, i)-{\rm th} \text{ entry}\\
		0 & \quad \text{otherwise}
	\end{cases} \quad \textrm{and} \quad
	\left[ \rho_{reg}(a_s) \right] = 
	\begin{cases}
		1 & \quad \textrm{for}~(l_{sj}, j)-{\rm th} \text{ entry}\\
		0 & \quad \text{otherwise}
	\end{cases},
$$
respectively. In fact, each $\left[ \beta_{reg}(a_s) \right]$ and $\left[ \rho_{reg}(a_s) \right]$ is a permutation matrix. The regular representation $\phi_{reg}: \Lambda_{A^{\op}} \rightarrow \GL(k\Lambda_{A^{\op}})$ is given by 
$$\phi_{reg}(a,b)(x, y) = (a,b)(x, y) = (a\cdot \lambda^{\op}_{b}(x), b\circ y )$$
for all $(a,b),(x,y) \in \Lambda_{A}.$
	Then $\beta_{\phi}: (A, \cdot) \rightarrow \GL(k\Lambda_{A^{\op}})$ is given by 
$$\beta_{\phi}(a_s) (a_i, a_j) = \phi_{reg}(a_s, 1)(a_i, a_j) = (a_s, 1)(a_i, a_j) = (a_s\cdot a_i, a_j) = (a_{k_{si}}, a_j)$$
	for each $1\leq i, j, s \leq n$. Similarly, $\rho_\phi:(A, \circ) \rightarrow \GL(k\Lambda_{A^{\op}})$ is given by
	\begin{align*}
		\rho_\phi(a_s) (a_i, a_j) & = \phi_{reg}(1, a_s)(a_i, a_j) = (1, a_s)(a_i, a_j) = (\lambda^{\op}_{a_s}(a_i), a_s \circ a_j)\\
		& = ((a_s \circ a_i)\cdot a_s^{-1}, a_{l_{sj}} ) = (a_{l_{si}}  \cdot a_{s}^{-1}, a_{l_{sj}}) = (a_{l_{si}^{\prime}} , a_{l_{sj}})
	\end{align*}
for some $1\leq l_{si}^{\prime}\leq n.$	Then,  for each $1\leq s \leq n$, $\beta_{\phi}(a_s)$ is a block diagonal matrix of order $n^2$, that is,
	$$[\beta_{\phi}(a_s)] = diag \big( [\beta_{reg}(a_s)],  [\beta_{reg}(a_s)], \ldots,  [\beta_{reg}(a_s)] \big).$$
	This gives $\beta_{\phi} =  |A| \cdot \beta_{reg}$. Similarly,  for each $1\leq s \leq n$, $\rho_\phi(a_s)$  is a block diagonal matrix of order $n^2$, that is, 
	$$\rho_{\phi}(a_s) = diag \big([\rho_{reg}(a_s)], [\rho_{reg}(a_s)], \ldots, [\rho_{reg}(a_s)] \big).$$
	Thus, we get $\rho_{\phi} =|A|\cdot \rho_{reg}$, which is desired. 
\end{proof}
\medskip

\section{Conjugacy classes and linearity of $\Lambda_{A^{\op}}$}\label{conjugacy and linearity}

Let $A$ be a skew left brace. Then  $(a_1, b_1), (a_2, b_2) \in \Lambda_{A^{\op}}$ are conjugate if there exists $(x,y) \in \Lambda_{A^{\op}}$ such that
\begin{align}
(a_1, b_1)=\;& (x,y) (a_2, b_2) (\lambda^{\op}_{y^{\dagger}}(x^{-1}), y^{\dagger}) \nonumber\\
=\;& (x \lambda^{\op}_y(a_2), y \circ b_2)(\lambda^{\op}_{y^{\dagger}}(x^{-1}), y^{\dagger}) \nonumber\\
=\; & (x \lambda^{\op}_y(a_2) \lambda^{\op}_{y \circ b_2 \circ y^{\dagger}}(x^{-1}), y \circ b_2 \circ y^\dagger).\label{conjrltn}
\end{align}

Let $\psi: \Lambda_{A^{\op}} \rightarrow \Aut(A, \cdot)$ be the group homomorphism defined by 
$$\psi(a,b)(x)=a \lambda^{\op}_b(x) a^{-1}$$
for all $a, b, x \in A$ (see  \cite{LV23}).  For $a \in A$, let $\mathcal{O}_a$ be the orbit of $a$ under the action $\psi$ and let
$$
\Fix(\psi) := \big\{ x \in A \mid \psi(a, b)(x) = x \text{ for all } a, b \in A \big\}.
$$

Note that $(A, \cdot) \times \{1\}$ is a normal subgroup of $\Lambda_{A^{\op}}$. By taking $b_2 = b_1 = 1$ in \eqref{conjrltn}, it follows that elements of the form $(a_1, 1)$ and $(a_2, 1)$ are conjugate in $\Lambda_{A^{\op}}$ if and only if $a_1$ and $a_2$ belong to same orbit under the action $\psi$. The following results are immediate.

\begin{prop}\label{prop conj clases} 
Let $A$ be a skew left brace and $a \in A$. Then the following assertions hold:
	\begin{enumerate}
\item $\Cl^{\Lambda_{A^{\op}}}(a, 1) = \mathcal{O}_a \times \{ 1\}$ for each $(a,1) \in \Lambda_{A^{\op}}$.
\item $ \{1 \}\times \Cl^{(A, \circ)}(a) 	\subseteq \Cl^{\Lambda_{A^{\op}}}(1, a)$. 
\item $\Fix(\psi)=\Z(A, \cdot) \cap \Fix(\lambda)$.
\item $\Ann(A) \trianglelefteq \Fix(\psi)$.
\item $\Ann(A) \times \Ann(A) \trianglelefteq \Ker(\psi) \trianglelefteq \Z( \Lambda_{A^{\op}}).$
	\end{enumerate}
\end{prop}

\begin{thm}\label{number of rep bound}
If $A$ is a finite skew left brace with $|A|>1$, then $$\max\{k(A, \circ), ~\textrm{the number of orbits of $\psi$} \} \leq k(\Lambda_{A^{\op}}) \leq k(A, \cdot)~k (A, \circ).$$
\end{thm}

\begin{proof}
It is well-known \cite{PXG70} that if $G$ is a group and $N \trianglelefteq G$, then $k(G) \leq k(G/N) ~ k(N)$. The assertion now follows from Proposition \ref{prop conj clases}.
\end{proof}

Recall that, a group $G$ is called linear if it admits a  finite-dimensional faithful linear representation over some field.

\begin{defn}
A skew left brace $A$ is said to be linear if the group $\Lambda_{A^{\op}}$ is linear.
\end{defn}

Clearly, if $A$ is linear, then both $(A, \cdot)$ and $(A, \circ)$ are linear. By Proposition \ref{prop:AcongAop}, if $A$ is a trivial skew left brace,  then $\Lambda_{A^{\op}}  \cong \Lambda_{A} \cong  (A, \cdot) \times (A, \cdot)$. In this case, the skew left brace $A$ is linear if and only if the group $(A, \cdot)$ is linear.

\begin{example}
{\rm Let $A$ be a left brace with the additive group being the group $(\mathbb{Z}, +)$ of integers. It is known due to Rump \cite[Proposition 6]{WR08} that $A$ is either the trivial left brace or its multiplicative group is given by
$$m \circ n = (-1)^m n + m$$
for all $m, n \in \mathbb{Z}$.  Clearly, $2\mathbb{Z}$ is a trivial sub-brace of $(\mathbb{Z},+, \circ)$ as well as a normal subgroup of $(\mathbb{Z}, \circ)$.  By the classification of infinite metacyclic groups \cite[Chapter 7]{CEH98}, we get $(\mathbb{Z}, \circ) \cong (\mathbb{Z}, +) \rtimes \mathbb{Z}_2$, which is linear being the infinite dihedral group. Thus, both the additive and the multiplicative groups of $A$ are linear.
\par

Note that $(\mathbb{Z}, +) \rtimes_{\lambda^{\op}} (2\mathbb{Z}, \circ) \cong (\mathbb{Z}, +) \times (\mathbb{Z}, +)$ is a normal subgroup of $\Lambda_{A^{\op}}$, with corresponding quotient group isomorphic to $\mathbb{Z}_2$, which leads to the short exact sequence
	$$0 \rightarrow (\mathbb{Z}, +) \rtimes_{\lambda} (2\mathbb{Z}, \circ) \xrightarrow{i} \Lambda_{A^{\op}}  \xrightarrow{\pi} \mathbb{Z}_2 \rightarrow 0.$$
It is easy to see that if a group admits a finite index subgroup which has a faithful linear representation, then the induced representation of the group is itself faithful.  Since free abelian groups admit faithful linear representations, the preceding short exact sequence implies that $\Lambda_{A^{\op}}$ is also linear.
}
\end{example}

We conclude with the following natural questions.

\begin{question}\label{lingrp}
Let $A$ be an infinite non-trivial skew left brace. 
\begin{enumerate}
\item When does the linearity of $(A, \cdot)$ imply the linearity of $(A, \circ)$ and vice versa? 
\item When does the linearity of both $(A, \cdot)$ and $(A, \circ)$ imply the linearity of $A$? Note that, the semi-direct product of two linear groups need not be linear in general.
\item What properties do linear skew left braces admit?
\end{enumerate}
\end{question}

\begin{ack}
The authors thank IISER Mohali for the institute fellowship during which this work was carried out. After the revision, the paper is solely authored by Rathee and Udeep. The authors also thank Mahender Singh for his valuable earlier contributions to the development of this work.
\end{ack}

\end{document}